\documentclass[12pt]{article}
\title{ Equivariant Principal Bundles over the 2-Sphere}
\author{ Eyup Yalcinkaya}
 
\date{\today}
\usepackage{tikz-cd}
\usepackage[all]{xy}
\usepackage{enumerate}
\usepackage{amssymb}
\usepackage{float}
\usepackage{setspace}
\usepackage{comment}
\usepackage{xcolor}
\usepackage{cancel}
\usepackage{amsmath,amsfonts,mathtools}
\usepackage{epsfig}
\usepackage{amsmath}
\usepackage{amsthm}
\usepackage{psfrag}

\usepackage{amsmath, amsthm, amssymb}
\usepackage[left=1.2in,top=.9in, 
right=1in,nohead]{geometry}
\newcommand{\be}{\begin{equation}}
\newcommand{\ee}{\end{equation}}

\newtheorem{theorem}{Theorem}[section]

\newtheorem{proposition}[theorem]{Proposition}

\newtheorem{corollary}[theorem]{Corollary}

\newtheorem{definition}[theorem]{Definition}

\newtheorem{remark}[theorem]{Remark}

\newtheorem{example}[theorem]{Example}

\def\ni{\noindent}

\begin{document}

\maketitle

\begin{abstract}

We classify $\Gamma$-equivariant principal $G$ bundles over $S^2,$ where $G$ is a compact connected Lie group and $\Gamma \subset SO(3)$ is finite. Restricted over the 1-skeleton of a carefully constructed $\Gamma$-equivariant CW decomposition, the bundle is determined by its restriction to singular points. The extension over $S^2$ is characterized by a Chern class. 
\vspace{.05in}

\ni {\em Keywords:} Equivariant principal bundles, Isotropy representation, Split bundles.

\vspace{.05in}

\ni {\em Mathematics Subject Classification 2010.} 
Primary 55R91, 55R15; Secondary 22A22.

\end{abstract}

\vspace{.05in}

\section{Introduction}

\indent Let $\Gamma$ and $G$ be Lie groups. A $\Gamma$-equivariant principal $(\Gamma, G)$-bundle $\xi$ over $X$ is locally trivial, principal $G$-bundle  $p: E \rightarrow X$ such that $E$ and $X$ are left $\Gamma$-spaces.  We denote the bundle $\xi=(E,X,p,G,\Gamma).$  The projection map $p$
is $\Gamma$-equivariant and $\gamma(e\cdot g) = (\gamma e)\cdot g$ 
where $\gamma \in \Gamma$ and $g \in G$ acting on $e \in E$. Equivariant principal bundles and their generalizations were studied by T. E. Stewart \cite {stewart}, T. tom Dieck \cite{tom}, R. Lashof \cite{lashof1} together with P. May \cite {lashof2} and G. Segal \cite{lashof4}.  These authors use homotopy-theoretic methods. There exists a classifying space $B(\Gamma, G)$ for principal $(\Gamma, G)$-bundles \cite{tom}, so  principal $(\Gamma, G)$-bundles over a $\Gamma$ space $X$ are studied by means of $[X,B(\Gamma,G)]_\Gamma$. If the structure group $G$ of the bundle is abelian, the main result of \cite{lashof4} states that equivariant bundles over a $\Gamma$-space $X$ are classified by ordinary homotopy classes of maps $[X\times_{\Gamma}E\Gamma, BG]$. Recently  M. K. Kim \cite{kim} has classified  equivariant complex bundles over $2$-sphere by means of CW decompositions of linear actions of $S^2$. Moreover, J. H. Verrette \cite{verrette2016results} studies  equivariant  vector  bundles  over the 2-sphere with effective actions by the rotational symmetries of the tetrahedron,  octahedron and  icosahedron for verifying Algebraic Realization Conjecture. 
\vspace{.01in}

\noindent Another approach to equivariant principal bundles is given by Hambleton and Hausmann \cite{ian1}. This approach proceeds the local invariants arising from isotropy representations at singular points of $(X,\Gamma)$.  By an isotropy representation at a $\Gamma$ fixed point $x_0 \in X,$ we imply the homomorphism $\alpha_{x_0} : \Gamma_{x_0} \rightarrow G$ defined by the formula 
$$ \gamma \cdot e_0 = \alpha(\gamma) \cdot e_0   $$
where $e_0 \in p^{-1}(x_0).$ 
Denote the collection of isotropy representation of the bundle $\xi$ by $Rep^G(\mathcal{I}).$ The homomorphism $\alpha$ is independent of the choice of $e_0$ up to conjugation in $G.$ In general, principal $(\Gamma, G)$-bundles over a $\Gamma$ space $X$ can not be classified by means of their approach however only on specific base spaces (i.e. split $\Gamma$-space).

\vspace{.1in}

\noindent In this article, we principally rely on the paper of Hambleton and Hausmann to determine $\Gamma$-equivariant principal $G$ bundles over $S^2$. It is proved that there exists a bijection between the equivalence classes of split $\Gamma$-equivariant $G$-bundles over $X$ and $Rep^G(\mathcal{I})$, provided that $X$ is a split $\Gamma$-space over $ A$ \cite{ian1}.

\noindent To begin with, the $\Gamma$-equivariant principal $G$-bundles  over split $\Gamma$-space $M$ are determined by its isotropy representation. However, it is obvious that $S^2$ is not a split $\Gamma$-space. On the other hand, 1-skeletons of $S^2$ are split $\Gamma$-space.

\noindent After determining the $\Gamma$-equivariant principal $G$-bundles over 1-skeletons lying on $S^2$, homotopy theoretic methods are next tool to separate $\Gamma$-equivariant principal $G$-bundles over $S^2.$ Let $\mathcal{A}^1 \subset S^2$ be a $\Gamma$-equivariant $CW$-complex, then
$$ \mathcal{A}^1 \xrightarrow{i} S^2 \xrightarrow{j} S^2 \cup C(\mathcal{A}^1) \xrightarrow{k} \Sigma(\mathcal{A}^1) \xrightarrow{\Sigma i} \Sigma(S^2) \xrightarrow{\Sigma j} \Sigma(S^2 \cup C(\mathcal{A}^1)) \rightarrow \cdots $$ is a cofibration sequence where   the cone over $\mathcal{A}^1$ is denoted by $C(\mathcal{A}^1)$  and  suspension of $\mathcal{A}^1$ is denoted by $\Sigma(\mathcal{A}^1)$.

\noindent Compute homotopy classes of maps into space $B(\Gamma, G)$, then the following sequence
$$ [\Sigma(S^2),Y] \xrightarrow{\Sigma i^*} [\Sigma(\mathcal{A}^1),Y] \xrightarrow{k^*} [S^2\cup C(\mathcal{A}^1), Y] \xrightarrow{j^*} [S^2,Y] \xrightarrow{i^*} [\mathcal{A}^1,Y] $$ is a exact sequence of abelian groups arising from homotopy classes of maps into space $B(\Gamma, G)$ since $B(\Gamma, G)$=$Y= \Omega Z$ is a loop space \cite{cost}. Besides, principal $G$-bundles over the 2-sphere are determined by the Steenrod equivalence theorem \cite[Th. 18.5]{steen}. The only invariant is the "Chern class"  $c(\xi) \in [S^2,BG]= \pi_2(BG)=\pi_1(G).$ This paper is devoted to proving the following theorem.


\begin{theorem}Let $\xi=(E,S^2,p,G,\Gamma)$ be a $\Gamma$-equivariant principal $G$-bundle over $S^2$ with a compact connected Abelian Lie group $G$ and $\Gamma\subset SO(3)$ be a finite subgroup acting on $S^2.$ A $\Gamma$-equivariant principal $G$-bundle over $S^2$ is determined by $Rep^G(\mathcal{I})$ and $c(\xi) \in \pi_2(BG).$
\end{theorem}
\begin{corollary} Let $\xi_1$ and $\xi_2$ be equivariant principal $G$-bundles over $S^2$. If $Rep^G(\mathcal{I}_{\xi_1})\cong Rep^G(\mathcal{I}_{\xi_2})$ then $c(\xi_1)\equiv c(\xi_2)$  $mod \ \lvert \Gamma\lvert.$ 
\end{corollary}

\vspace{2in} 

\section{Preliminaries}
\indent In this section, some definitions and
theorems are presented from the book of T.  tom Dieck \cite{tom}.   In this paper, we consider the group action $\Gamma$ as a left action. Let $X$ be a topological space and $\Gamma$ be a topological group acting on $X.$ We consider the isotropy group (stabilizer group) of each point in $X$ and the orbit space of the space $X$ under the group action. For each $x\in X$, the \emph{singular points set} of $X$ is denoted by $Sing(X,\Gamma) = \{ x \in X \lvert \ \Gamma_x \neq 1 \}$  if isotropy subgroup of $x$ is not identity.
An action is called \emph{transitive} if for every $x_1,x_2$ $\in$ $X$ there exists an element $\gamma\in \Gamma$ such that $\gamma x_1$ = $x_2$ and the action is called \emph{free} if for every $x$ $\in$ $X$ the only element of $\Gamma$ (identity) fixes the point $x.$

\vspace{.1in}

\noindent A simplicial $\Gamma$-complex is \emph{regular} if for arbitrary $\gamma_i \in \Gamma$  and two simplices $(v_1, v_2, ... , v_n)$ and $(\gamma_1 \cdot v_1, \gamma_2 \cdot v_2, ... , \gamma_n \cdot v_n)$, there is an element $\gamma\in \Gamma$ such that $\gamma v_i={\gamma_i}{v_i}$ for all i.
In general, simplicial complexes may not be regular. However, the simplicial complex that is not regular can be constructed as a regular by means of a barycentric subdivision.

\vspace{.1in}

\noindent The \emph{Riemann-Hurwitz Formula} is generally the tool of algebraic geometry. In this context, the analogous result of this formula for the graph is the following;

\begin{proposition} \label{RHF}
Let $X$ be a compact connected regular simplicial graph, $\Gamma$ be a finite group and $\Gamma$ $\times$ $X$ $\rightarrow$ $X$ be a regular left transitive group action with the orbit space $X /\Gamma \cong A.$
Then $$\chi(X)= |\Gamma| \chi(A)-\sum_{v_i\in  X}(|\Gamma_{v_i}|-1)$$
\end{proposition}

\begin{definition}[Split $\Gamma$-Space] Let $\Gamma$ be a Hausdorff topological group and $A$ be a topological space. A $\Gamma$-space is a topological
space equipped with a continuous left action of $\Gamma$. If the space $X$ is a $\Gamma$-space and $x \in X,$ we denote
$\Gamma_{x}$ as the stabilizer of $x.$
A \emph{ split $\Gamma$-space } over $A$ is denoted by a triple
$(X, \pi, \varphi)$ where
\begin{itemize}
\item $X$ is a $\Gamma$-space.
\item $\pi: X \rightarrow A$ is a continuous surjective map and, for each $a\in A$, the preimage
$\pi^{-1}(a)$ is a single orbit.
\item $\varphi$: A $\rightarrow$ X is a continuous section of $\pi$
\end{itemize}
\end{definition}
\noindent In this definition, we might prefer a \emph{ split $\Gamma$-space }  over the space $A$ by omitting the notation of maps $\pi$ and $\varphi$. Also, the map $\pi$ induces
$\bar{\pi}$: $ X  \slash \Gamma $ $\rightarrow$ $A$ which is a homeomorphism since $\varphi$ provides its continuous inverse and $\Gamma $ is Hausdorff.

\subsection{Split Equivariant Principal Bundles}
\noindent  Let $X$ be a split $\Gamma$-space. The isotropy groupoid and representation of $X$ can be defined arising from the group $\Gamma$ and the orbit space.

\begin{definition}[Isotropy Groupoid] Let (X,$ \pi$,$\varphi$) be a split $\Gamma$-space over the orbit space A. The \textbf{Isotropy groupoid} of $X$ is denoted by
$$\mathcal{I}(X):=\{(\lambda, a)\in \Gamma \times A \lvert \ \lambda \in \Gamma_{\varphi(a)}\} $$
\end{definition}
\noindent The isotropy groupoid of the space $X$ is the subspace of $\Gamma \times A$ such that for each $a\in A$, the space
$\mathcal{I}_a= \mathcal{I} \cap(\Gamma \times \{a\})$ can be written as the form $\mathcal{I}_a' \times \{a\},$ where $\mathcal{I}_a'$
is a closed subgroup of $\Gamma.$ As a short notation, $\mathcal{I}$ is used for the isotropy representation of  split $\Gamma$-space of $X$ with the orbit space $A.$ A groupoid $\mathcal{I}$ is called locally maximal if each point $a \in A$ admits a
neighbourhood $U$ such that u is a subgroup of $\mathcal{I}_a$ for all $\mathcal{I}_u \in U$. 

\indent We denote by $A^{(n)}$ the skelata of $A,$ by $\Omega$ = $\Omega(A)$ the set of cells of A,  by $d(e)$ the dimension of a cell $e$ $\in$ $\Omega$ and $\Omega_n$ = $\{$ $e$ $\in$ $\Omega$ $\lvert$ $d(e)$ = $n$ $\}$  for the CW-complex $A.$ Also, we denote by $e(a)$ $\in$ $\Omega$ the cell $e$ of the smallest dimension such that $a$ $\in$ $A.$
\begin{definition} A $(\Gamma, A)$ groupoid $\mathcal{I}$ is called cellular if it is locally maximal and if $\tilde{\mathcal{I}}_a = \tilde{\mathcal{I}}_b$ when $e(a) = e(b).$
\end{definition}
\begin{definition}[Split Bundle] Let $(X, \pi,\varphi)$ be a split $\Gamma$-space over $A$ with isotropy groupoid $\mathcal{I}$ and let $\xi=(E,X,p,G,\Gamma)$ be a $\Gamma$-equivariant principal $G$-bundle over $X$. Then, the bundle $\xi$ is called a \textbf{split bundle} if $\varphi^*\xi$ is trivial.
\end{definition}

\noindent $\Gamma$-equivariant principal $G$-bundle is a split bundle if the orbit space A is a contractible and paracompact space \cite{ian1}.

\begin{definition}[Isotropy Representation] Let $\xi=(E,X,p,G,\Gamma)$ be a $\Gamma$-equivariant principal $G$-bundle over $X$ and (X,$ \pi$,$\varphi$) be a split $\Gamma$-space over the orbit space $A$ and $G$ be a topological group and $\mathcal{I}$ be a $(\Gamma,A)$ groupoid of $X$. A continuous representation of $\mathcal{I}$ is continuous map
$$\iota : \mathcal{I}\rightarrow G$$ such that the restriction of $\iota$ to each point $a$ $\in$ $A$ is group homomorphism from $\mathcal{I}_a$ to $G$ and denoted by
$\iota_a$: $\mathcal{I}_a$ $\rightarrow$ $G$.

\vspace{.1in}
\noindent A continuous representation of $\iota :\mathcal{I} \rightarrow G$ is called locally maximal if for each point $a \in A$, there exists a neighborhood $U$ such that $\mathcal{I}_u$ is subgroup of $\mathcal{I}_a$ for all $u \in U$ and 
 cellular if $\iota_a = \iota_b$ when $e(a) = e(b).$
Moreover, the isotropy groupoid $\mathcal{I}$ is called weakly locally maximal if there exists a continuous map $g: U \rightarrow G$ such that $\alpha_u (\gamma) = g(u) \alpha_a (\gamma) g(u)^{-1}$ for all $u\in U$ and $\gamma \in \mathcal{I}_u.$
 \end{definition}
\noindent The isotropy representation of  $\mathcal{I}$ is the continuous groupoid representation of $\mathcal{I}$ in $G$. In fact, it is well-defined up to conjugation by $Map(A, G)$.
 The set of conjugacy classes of locally maximal continuous representations of $\mathcal{I}$ can be denoted $$Rep^{G}(\mathcal{I})= Hom(\mathcal{I},G)/Map(A,G).$$

\noindent Let $(X, \pi,\varphi)$ be a split $\Gamma$-space with isotropy groupoid $\mathcal{I}$ and $A$ be an orbit space of $X$ by group action $\Gamma.$ Suppose  $\xi$ be a split $\Gamma$-equivariant principal $G$-bundle over the space $X$ then there exists a continuous lifting $\tilde{\varphi}^*(\xi) : A \rightarrow E$ of $\varphi$ since $\varphi^*(\xi)$ is trivial.  The equation
$$
\gamma \tilde{\varphi} (a) = \tilde{\varphi}(a) \alpha_a(\gamma),$$
(valid for $a \in A$ and $\gamma \in \mathcal{I}_a)$ determines a continuous representation $\alpha_{\xi,\tilde{\varphi}} : \mathcal{I} \rightarrow G$
which does not depend on the choices $\tilde{\varphi}$ and depends only on the $\Gamma$-equivariant isomorphism class of $\xi$ \cite{ian1}.

\vspace{0.1in}
\noindent In this case, the class of $\Gamma$-equivariant principal $G$-bundles over Split $\Gamma$-space $X$ is denoted by $SBun_\Gamma^G.$ Then, the following theorem states the bijection between this class and isotropy representation of $\mathcal{I}.$
\begin{theorem} \cite{ian1} \label{class}
Let $(X,\pi,\varphi)$ be a split $\Gamma$-space over the orbit space $A$ with the isotropy groupoid $\mathcal{I}$ of $X$. Assume that $A$ is locally compact, the group $\Gamma$ is a compact Lie group, and $\mathcal{I}$ is locally maximal. Then for any compact connected Lie group G, the map
$$ \Phi: SBun_{\Gamma}^{G} \rightarrow Rep^{G}(\mathcal{I})$$ is a bijection.
\end{theorem} 

\noindent Specifically, if the orbit space is contractible, equivariant bundles turn out split. Except for this case, the following  proposition is another result for equivariant bundles provided that the structural group $G$ of the bundle is abelian.
 
\begin{proposition} \label{prop} \cite{ian1} Let $\Gamma$ be a compact Lie group and let 
$(X,\pi,\varphi)$ be a split $\Gamma$-space over the orbit space $A$ with the isotropy groupoid $\mathcal{I}.$
Suppose that $\mathcal{I}$ is locally maximal and that $A$ is a locally compact space. If the group $G$ is a compact connected abelian group, then there exists an isomorphism between abelian groups
$$ (\Phi,\varphi*) : Bun^G_\Gamma(X) \rightarrow Rep^G(\mathcal{I}) \times Bun^G(A).$$
\end{proposition}

\section{Equivariant Principal $G$-bundles}

\subsection{1-skeletons on $S^2$} \label{section:a}

\noindent The 2-sphere is not a split $\Gamma$-space. On the other hand, 1-skeletons lying on  the 2-sphere can be constructed as split $\Gamma$-space providing that 1-skeletons are regular simplicial $\Gamma$-complex.
 
  \vspace{.05in}

\begin{theorem}\cite{rees2005notes}  Let $\Gamma$ be a finite subgroup of $SO(3)$. Then $\Gamma$ is isomorphic to precisely one of
the following groups: \begin{enumerate}[i)]

\item $\mathbb{Z}_n,$  $(n \geq 1)$: rotational symmetry group of an n-pyramid
\item $D_n$, $(n \geq 2)$: rotational symmetry group of an n-prism
\item $A_4$: rotational symmetry group of a regular tetrahedron
\item $S_4$: rotational symmetry group of a cube (or a regular octahedron)
\item $A_5$: rotational symmetry group of a regular dodecahedron (or a regular icosahedron).
\end{enumerate}
\end{theorem}
 \indent For each finite subgroup of $SO(3)$, we determine $\Gamma$-equivariant 1-skeleton $\mathcal{A}^1 \subset S^2$ and the orbit space $\mathcal{A}^1/ \Gamma\simeq A. $

\begin{theorem}\label{splittheorem}
 Let $\Gamma \subset SO(3)$ be a finite subgroup. Then, for each subgroup $ \Gamma$, there exists a $\Gamma$-equivariant 1-skeleton $\mathcal{A}^1 \subset S^2$ such that the 1-skeleton $\mathcal{A}^1$ is a split $\Gamma$-space over the orbit space A.
\end{theorem} 
\begin{proof}
\noindent For each case, $\Gamma$-equivariant 1-skeletons can be constructed on the $S^2$ such that the Riemannian-Hurwitz formula is satisfied. 1- skeletons for cyclic and dihedral cases are inductively constructed. By this induction, it can be shown that each group $\Gamma$ acting on 1-skeletons yields fixed orbit spaces. For the other subgroups of the $SO(3)$, it can be shown by direct computation for special $\Gamma$-equivariant CW complex $\mathcal{A}^1$. Given any subgroup $\Gamma\subset SO(3)$, 1-skeleton $\mathcal{A}^1$ and its orbit space $\mathcal{A}^1/\Gamma \cong A $ must satisfy Riemann-Hurwitz formula. Then, the 1-skeleton $\mathcal{A}^1$ satisfies the condition of being split $\Gamma$-space.

\begin{enumerate}[(i)]

 \item Cyclic Subgroups

Let $\mathfrak{C_n}$ be a 1-skeleton lying on $ S^2$ such that it contains $2$ vertices at north and south poles and $n$ edges longitudinal semicircles through the points                     $$(\cos(2\pi k/ n), \sin(2\pi k/n))$$ joining the poles
for $0 \leq k < n$ and $ \mathbb{Z}_n $ be a cyclic group with n elements acting on $\mathfrak{C_n}.$

\vspace{0.1in} 
Let $\mathbb{Z}_n$ be the cyclic group acting on $\mathfrak{C_n}$, with the orbit space $E_n \cong \mathfrak{C_n}/\mathbb{Z}_n.$ Hence, the CW-complex $\mathfrak{C_n}$ becomes a split $\mathbb Z_n$-space. Then, the Riemann-Hurwitz Formula must be satisfied for the $E_n \cong \mathfrak{C_n}/ \mathbb{Z}_n$; 
\begin{displaymath} \chi( \mathfrak{C_n})=n\chi(E_n)-\sum_{p\in \mathfrak{C_n}}(|\mathbb {Z}_{n_{p}}|-1)
\end{displaymath}
 where $n_p$ is the order of the isotropy subgroup of the point $p.$ Therefore Euler Characteristic;
$$ \chi( \mathfrak{C_n})=2-n\ \text{and}  \ \chi(E_n)=1.
$$
We claim that
$$ \chi( \mathfrak {C_{n+1}})=\chi( \mathfrak{C_n})-1 \ \text{and} \ \chi(E_{n+1})=\chi(E_n).
$$

\noindent After attaching one new edge to the $\mathfrak{C_n}$, it turns out to be $\mathfrak{C_{n+1}}.$ Conversely, the orbit space $E_{n+1}$ stays the same as the $E_n$.

Let $n=2$, $\Gamma=\mathbb{Z}_2$, then $\chi(\emph{ $\mathfrak{C_2}$})=0$, $\chi(E_2)=1,$ and the Riemann-Hurwitz Formula holds.\\
Suppose $\chi(\mathfrak{C_{n+1}})=1-n$ since the CW-complex  $\mathfrak{C_{n+1}}$ has 2 vertices and
$n+1$ edges. Therefore, $\chi(\mathbf{\mathfrak{C_{n+1}}})=\chi(\mathbf{ \mathfrak{C_n}})-1$, similarly $\mathbf{ \mathfrak{C_{n+1}}} /\mathbb{Z}_{n+1}\cong E_{n+1}$ satisfies the
Riemann-Hurwitz Formula.

\begin{figure}[H]
\centering

\includegraphics[height=20mm, width=100mm,]{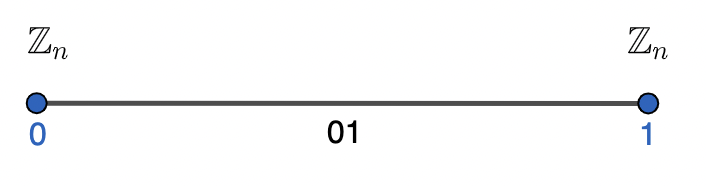}

\caption{Orbit spaces of cyclic groups}
\end{figure}
Let $\mathcal{I}$ denote the isotropy groupoid of the CW-complex $\mathfrak{C_n}.$  It is cellular since it is defined by $\mathcal{I}_0$=$\mathcal{I}_1$=$ \mathbb{Z}_n $ and $\mathcal{I}_{01} =id$. Hence, it can be concluded that there is a continuous section from $E_n$ to $\mathfrak{C_n}.$

 \item Dihedral Subgroups

Let $\mathfrak D_n$ be a 1-skeleton lying on $S^2$ with $2n+2$ vertices and $6n$ edges. The vertices of $\mathfrak D_n$ are as follows:
\begin{itemize}
\item The vertices of the $n$-gon on the equator
\item The middle points of the edges of the $n$-gon
\item The south and north poles.
\end{itemize}
In this case, the 1-skeleton turns out to be a regular simplex. The 1-skeleton contains $6n$ edges: longitudinal quarter circles through the points
$$(\cos(\pi k/ n), \sin(\pi k/n))$$ joining the poles for $0 \leq k < n$ and transversal edges on the equator at $\cos(\pi k/n), \sin (\pi k/ n)$.

\vspace{0.1in}

\noindent Let  $D_n$ be the dihedral group acting on $\mathfrak D_n$ with the orbit space $\mathcal D_n\cong  \mathfrak D_n/ D_n.$
Hence, the CW-complex $\mathfrak D_n$ becomes a split $D_n$-space. Inductively, we show that the orbit space $\mathcal D_n$ and the CW-complex $\mathfrak D_n$ satisfy the  Riemann-Hurwitz Formula;
\begin{displaymath} \chi(\mathfrak D_n)=2n\chi(\mathcal D_n)-\sum_{p\in \mathfrak D_n}(|D_{{2n}_p}|-1)
\end{displaymath} 
where $n_p$ is the order of isotropy subgroup of the point $p$ and hence Euler Characteristic;
\begin{displaymath} \chi(\mathfrak D_n)=(2n+2)-6n=2-4n \ \text{and}\ \chi(\mathcal D_n)=0.
\end{displaymath}
\noindent Attaching a vertex to the CW-complex $\mathfrak D_n$ yields the new CW-complex $\mathfrak D_{n+1}$. Then, there is a relation between the Euler characteristics of these CW-complexes and their orbit spaces as follows;
\begin{displaymath} \chi(\mathfrak D_{n+1})=\chi(\mathfrak D_n)-4 \  and \  \chi(\mathcal D_{n+1})=\chi(\mathcal D_n).
\end{displaymath}

\noindent Let $n=4$, $\Gamma=D_4$ and $\mathfrak D_4$ be a 1-skeleton lying on $S^2$ such that it has a square stating on the equator of
$S^2$.
\indent Let the vertices of the square be labeled 1, 2, 3, 4. \newline
$$\Gamma=D_4= \{ (),(1234),(1423),(1432),(14)(23),(12)(34),(13),(24) \} $$
then \begin{displaymath} \chi(\mathfrak D_4)=-6 \ \ \ \ \ \ \chi(\mathcal D_4)=0 \end{displaymath}
and \begin{displaymath} \chi(\mathfrak D_n)=2-4n
\end{displaymath}
\indent $\chi(\mathfrak D_{n+1})=-2-4n$  with $ 4n+6$ vertices and
$8n+8$ edges.
The relation between ewuler characteristics of $\mathfrak D_{n+1}$ and $\mathcal D_{n+1}$ as follows;
\begin{displaymath} \chi(\mathfrak D_{n+1})=\chi(\mathfrak D_n)-4.
\end{displaymath} 

Hence, $\mathfrak D_{n+1}$ and $\mathcal D_{n+1}$ satisfy the
Riemann-Hurwitz Formula.

\begin{figure}[H]
\centering

\includegraphics[height=40mm, width=80mm,]{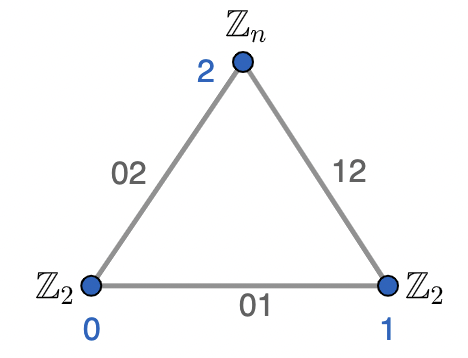}
\caption{Orbit space of the dihedral group}
\end{figure}

The isotropy groupoid $\mathcal{I}$ on the orbit space is calculated as $\mathcal{I}_0$=$\mathcal{I}_1$=$ \mathbb{Z}_2 $,\ $\mathcal{I}_2= \mathbb{Z}_n $ and $\mathcal{I}_{01} = \mathcal{I}_{12}=\mathcal{I}_{02} = id.$ Hence, there exists a continuous section from orbit space to $\mathfrak D_n$.

\item Tetrahedral Subgroup

\noindent Let $\mathfrak T$ be an 1-skeleton tetrahedron lying on $S^2$, 
 It is necessary to add 6 vertices
in the center of edges and 4 vertices in the center of the faces to obtain a regular CW-complex. Hence, $\mathfrak T$ is regular with 14 vertices and 24 edges.

$$\chi(\mathfrak T)=-10$$

Let $A_4$ be the tetrahedral group order 12 acting on $\mathfrak T$ with the orbit space $T\cong \mathfrak T/ A_4$.

Hence, the tetrahedron $\mathfrak T$ has 4 vertex-rotation with the order 3 
and 3 edge-rotation with the order 2. 
Hence, the tetrahedron $\mathfrak T$ is split $A_4$-space since the Riemann-Hurwitz Formula holds as follows;
$$-10=12\chi(T)-\sum_{p\in \mathfrak{T}}(|\mathfrak{T} _p|-1)$$
$$-10=12\chi(T)-(4(2+2)+3(1+1)) $$

$$\chi(T)=1$$

The orbit space $T\cong \mathfrak T/ A_4$ is the following;
\begin{center}
\includegraphics[height=20mm, width=100mm,]{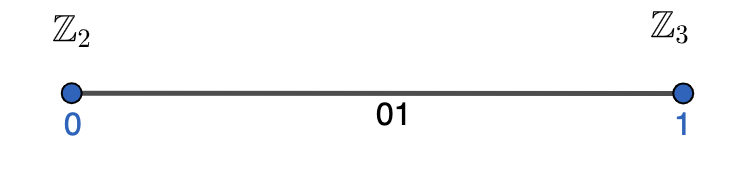}
\end{center}
We can say that the isotropy groupoid $\mathcal{I}$ of $\mathfrak{T}$ is cellular and it is given by $\mathcal{I}_0$=$ \mathbb{Z}_2 $ $\mathcal{I}_1$=$ \mathbb{Z}_3 $ and $\mathcal{I}_{01}$ $=$id. Hence, we can find a continuous section from orbit space to $\mathfrak T.$
\item Octahedral Subgroup

\indent Let $\mathcal C$ be a 1-skeleton cube lying on $S^2.$ 
Firstly, the cube $\mathcal C$ has 8 vertices and 12 edges.
It is necessary extra 12 vertices in the center of edges
and 6 vertices in the center of faces. Hence, the cube $\mathcal C$ is regular with 26 vertices.
$$\chi(\mathcal C)=-22$$

Let $S_4$ be an octahedral group order 24 acting on the cube $\mathcal C$ with the orbit space $O \cong \mathcal C /S_4$.
The cube $\mathcal C$ has 4 vertex-rotation of order 3,
6 edge-rotation of order 2,
and 3 face-rotation of order 4.
Therefore, the orbit space $O$ is composed of 3 vertices and 2 edges. The cube
$\mathcal C$ is a split $S_4$-space since the Riemann-Hurwitz Formula holds as follows;
$$\chi(\mathcal C)=24\chi(O)-\sum_{p\in \mathcal C}(|\mathcal C _p|-1)$$

$$-22=24\chi(O)-(3(3+3)+4(2+2)+6(1+1))$$
$$\chi(O)=1$$
\begin{center}
\includegraphics[height=50mm, width=100mm,]{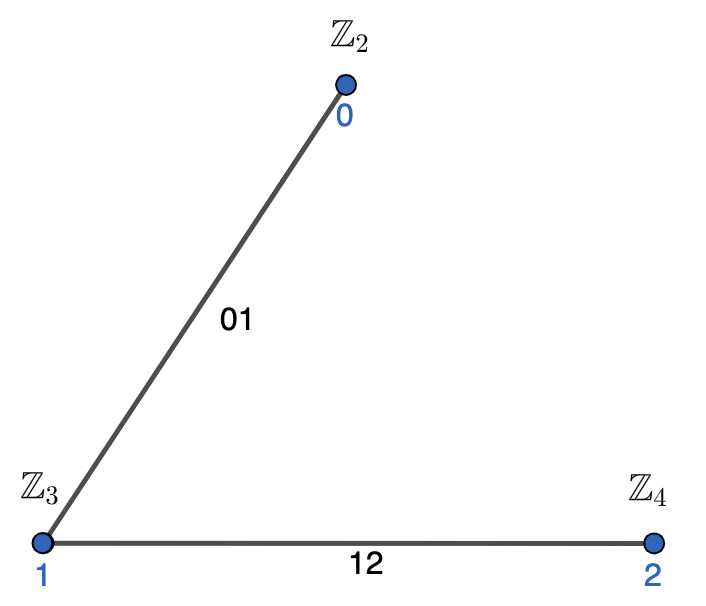}
\end{center}
We can say that the isotropy groupoid $\mathcal{I}$ of the cube $\mathcal {C}$ is  cellular since it is given by $\mathcal{I}_0$=$ \mathbb{Z}_2 $, $\mathcal{I}_1$=$ \mathbb{Z}_3 $, $\mathcal{I}_2$=$ \mathbb{Z}_4 $, $\mathcal{I}_{01}$ $=$id and $\mathcal{I}_{12}$ $=$id. Hence, we can construct a continuous section from orbit space to $\mathcal C.$


\item Icosahedral Subgroup

\indent Let $\mathcal O$ be an 1-skeleton Icosahedron lying on $S^2.$  
Firstly, the Icosahedron $\mathcal O$ has 12 vertices and 30 edges. 
By adding 30 vertices in the center of edges
and 20 vertices in the center of faces, the regular Icosahedron $\mathcal O$ has 62 vertices and 120 edges.
$$\chi( \mathcal O)=-58$$

Let $A_5$ be an icosahedral group order 60 acting on $\mathcal O$ with the orbit space $I\cong \mathcal O / A_5.$
The Icosahedron $\mathcal O$ has 6 vertex-rotation of order 5,
15 edge-rotation of order 2,
and 10 face-rotation of order 3.
Then, the orbit space $I$ is composed of 3 vertices and 2 edges.
Therefore, the Icosahedron $\mathcal I$ is a split $A_5$-space since it satisfies the Riemann-Hurwitz Formula at proposition \ref{RHF}.
$$\chi( \mathcal O)=60\chi(I)-\sum_{p\in \mathcal O}(|\mathcal O _p|-1)$$
$$-22=60\chi(I)-(6(4+4)+15(2+2)+10(1+1))$$
$$\chi(I)=1$$
\begin{center}
\includegraphics[height=50mm, width=100mm,]{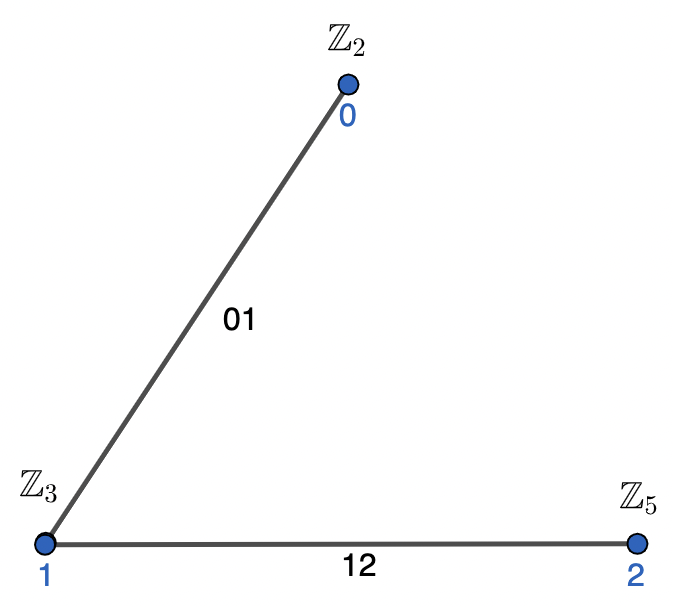}
\end{center}

we can say that the isotropy groupoid $\mathcal{I}$ of the Icosahedron $\mathcal O$ is cellular and it is given by $\mathcal{I}_0= \mathbb{Z}_2 $, $\mathcal{I}_1 = \mathbb{Z}_3 $, $\mathcal{I}_2= \mathbb{Z}_5 $, $\mathcal{I}_{01}=$id and $\mathcal{I}_{12}=$id. Then we can construct a continuous section from orbit space to $\mathcal O.$
\end{enumerate}

\noindent Therefore, for each finite subgroup $\Gamma \subset SO(3)$, we construct $ \Gamma$-equivariant split 1-skeleton lying on $S^2$.

\end{proof}

\indent  Now, we show that for each finite subgroup 
$\Gamma\subset SO(3),$ a $\Gamma$-equivariant 1-skeleton CW-complex $\mathcal{A}^1$ on $S^2$ can be constructed such that the CW-complex $\mathcal{A}^1$ splits over the orbit space $A$.
\subsection{Classification of $G$-bundles Over 1-Skeletons On $S^2$}
\label{section:Classification of Principal $G$-bundle over 1-Skeletons on $S^2$}
\noindent

\noindent The class of isotropy representations, $Rep^G(\mathcal{I})$, can be quite complex. However, by constructing specific regular 1-skeletons of $S^2$ for each subgroup of $SO(3)$, one can obtain a split $\Gamma$-space. The orbit spaces of cyclic groups, the tetrahedral group, the octahedral group, and the icosahedral group are homeomorphic to the interval [-1,1]. Additionally, every $\Gamma$-equivariant principal $G$-bundle can be considered a split $\Gamma$-space, where the orbit space $A$ is both contractible and paracompact. In the case of the dihedral group, the orbit space A is a triangle (i.e. graph), and there exists a bijection between the split bundle space of the CW-complex $\mathcal{A}^1$ and the isotropy representation of groupoid $\mathcal{I}$. When the group $G$ is abelian, there is an isomorphism between the bundle spaces of $\mathcal{A}^1$ and $Rep^G(\mathcal{I})$
$\times$ $Bun^G(A).$

\vspace{.05in} 
\noindent We can deduce that all the orbit spaces are paracompact, as they are all compact. As a result, for any finite subgroup $\Gamma \subset SO(3)$, the orbit spaces can be classified as either contractible or non-contractible. 





\subsubsection{Contractible Case}
\noindent All of the equivariant bundles over the space $\mathcal{A}^1$ are split bundles, as the orbit space of $\mathcal{A}^1$ is both contractible and paracompact. The specific spaces $(*$,$\pi$,$\varphi$ $)$ can be split over their orbit spaces, as stated in Theorem \ref{class}. Additionally, all possible orbit spaces are locally compact, and the possible isotropy groupoids are locally maximal. Since the group $\Gamma$ is a compact Lie group, the map
\begin{equation} \label{ceq} \Phi: SBun_{\Gamma}^{G} \rightarrow Rep^{G}(\mathcal{I}) \end{equation}
is a bijection.
\subsubsection{Non-contractible Case}
In the same way, the CW-complex $\mathfrak D_n$ splits over the orbit space $\mathcal D_n$ and the map \ref{ceq} is a bijection. However, it should be noted that not all equivariant principal $G$ bundles are split over the space $\mathcal{A}^1$. To handle the non-split bundles over the space $\mathcal{A}^1$, we restrict $G$ to be abelian in order to meet the requirement of Proposition \ref{prop}.
The space $(\mathfrak D_n,\pi,\varphi)$ is a split $D_n$-space over $\mathcal D_n$ with the isotropy groupoid $\mathcal{I}$, and $\mathcal D_n$ is locally compact and $\mathcal{I}$ is locally maximal. Therefore, for any abelian Lie group $G$, the following map is an isomorphism;
$$(\Phi,\varphi*) : Bun^G_\Gamma(\mathcal{A}^1) \rightarrow Rep^G(\mathcal{I}) \times Bun^G(A). $$
Since the orbit space $\mathcal D_n$ is a triangle and is homeomorphic to $S^1$, principal $G$-bundles over $S^1$ can be induced by the following map
\begin{equation} \label{deq} \overline{(\Phi,\varphi*)} : Bun^G_\Gamma(\mathcal{A}^1) \xrightarrow{\approx} Rep^G(\mathcal{I}) \times Bun^G(S^1). \end{equation}
In conclusion, for a connected compact Lie group $G$, the following holds:
$$ [S^1,BG] \cong \pi_{1}(BG) \cong \pi_{0}(G) \cong 0. $$
In the case of a disconnected group $G$, the analogy is as follows:
$$ [\mathcal{A}^1,S^1] \cong \pi_{1}(BG) \cong \pi_{0}(G) $$ provided that group $G$ is compact.
\section{Calculation of $Rep^G(\mathcal{I})$}

Let $\iota$ : $\mathcal{I}$ $\rightarrow$ $G$ be an isotropy representation and let $\mathcal{I}$ be a $(\Gamma$,$A)$-groupoid.  For each $e\in \Omega(A),$ this defines Hom($\mathcal{I}_e$, $G)$ with face compatibility conditions $\iota_e$ = $\iota_f$ $\lvert$ $\mathcal{I}_e$ whenever $f$ $\le$ $e.$ The set of conjugacy classes of cellular representations of $\mathcal{I}$ into $G$ is denoted by $Rep^{G}_{cell}$ $(\mathcal{I}).$

\vspace{0.05in}

\indent  For a cellular representation $\iota$ : $\mathcal{I}$ $\rightarrow$ $G$ and a cell $e \in \Omega(A),$ the associated conjugacy class of $\iota$ is denoted by $[\iota_e] \in \overline{Hom} (\mathcal{I},G)$. Hence, the following map $\beta$ is well-defined; 
$$ \beta : Rep_{cell}^G(\mathcal{I}) \rightarrow \prod_ {e \in \Omega (A)} {\overline{Hom}(\mathcal{I}_e, G).}$$
 For each $b_e \in  \prod_ {e \in \Omega (A)} \overline{Hom}(\mathcal{I}_e, G)$, faces compatible to each other. Then, we define 

$$ \overline{ Rep}_{cell}^G(\mathcal{I}) = \{(b_e) \in \prod_ {e \in \Omega (A)} \lvert \  b_e = b_f \lvert \mathcal{I}_e \ \text{if}\ f \le e \}$$ and we can replace $\beta$ as a map $\bar{\beta}$ : $Rep_{cell}^G(\mathcal{I})$ $\rightarrow$ $\overline{Rep}_{cell}^G(\mathcal{I}).$ Then the diagram is commutative,

$$\begin{tikzcd}[column sep=small]
Rep_{cell}^G(\mathcal{I})\arrow{rr}{\tau} \arrow[swap]{dr}{\bar{\beta}}& &Rep^G(\mathcal{I}) \arrow{dl}{\upsilon}\\
& \overline{Rep}_{cell}^G(\mathcal{I}) & 
\end{tikzcd}    $$
when $\mathcal{I}$ is proper $(\Gamma$,$A)$-groupoid.

The map $\tau$ is obvious since a cellular representation is a representation, which is locally
maximal. To define $\upsilon (\beta)e$ for 
$e \in \Omega(A)$, we choose $a \in A$ with $e(a) = e$ and set $\upsilon(\beta)e = [\beta a]$.
Since cells are connected, $\upsilon$ is well defined. On the other hand, none of these maps is either surjective or injective in general.

\begin{theorem} Let $\mathcal{A}^1\subset S^2$ and G be a topological group. Let $\Gamma$ $\subset$ $SO(3),$  $A=\mathcal{A}^1/\Gamma$ be an orbit space and $\mathcal{I}$ be a $(\Gamma$,$A)$-groupoid. Then $\bar{\beta} : Rep_{cell}^G(\mathcal{I}) \rightarrow \overline{Rep}_{cell}^G(\mathcal{I})$
is surjective.
\end{theorem}
\begin{proof} For a finite subgroup $\Gamma\subset SO(3),$ the orbit spaces of subgroups except dihedral group are tree.  Hambleton and Hausmann prove this theorem provided that the orbit space is tree \cite{ian2}. Now, we only prove the case for the dihedral group. Let $b \in \overline{Rep}_{cell}^G(\mathcal{I})$ and let $v$ be a vertex of $A.$ We choose $\iota_v \in Hom(\mathcal{I}_v,G)$ representing $b_v.$ For an edge $e$ between $v$ and $v'$ we define $\iota_e$ $\in$ $Hom(\mathcal{I}_e,G)$ by $\iota_e = \iota_v \lvert  \mathcal{I}_e.$ Since $b \in \overline{Rep}_{cell}^G(\mathcal{I})$, we choose $\iota_{v'} \in Hom(\mathcal{I}_{v'},G)$ where $\iota_{v'}=\iota_e.$ Therefore, we define a cellular representation $\iota_{v,1}$ over the tree $A(v,1)$ of the points of distance smaller than or equal to 1 far from $v.$ We construct $\iota_{v,2}$ over $A(v,2)$ with the same way. Hence, we choose the points of distances smaller than 3 far from $v,$ when $A(v,3)$ is defined. Now, this defines $\iota \in Rep_{cell}^G(\mathcal{I})$ with $\bar{\beta}(\iota) = b.$
\end{proof}
\begin{proposition} \cite{ian1}
Let $\mathcal{I}$ is a $(\Gamma,A)$-groupoid, where $\Gamma$ $\subset$ $SO(3)$ and the orbit space $A$ is a graph. Let $G$ be a path-connected topological group. Then $\upsilon : Rep^G(\mathcal{I})\rightarrow \overline{Rep}_{cell}^G(\mathcal{I})$ is surjective.
\end{proposition}

\begin{theorem}
Let $\mathcal{I}$ be a proper $(\Gamma,A)$-groupoid with  $\Gamma \subset SO(3)$ a finite topological group and let $A$ be an orbit space. Let $G$ be a compact connected Lie group. Then $\tau: Rep_{cell}^G(\mathcal{I}) \rightarrow Rep^G(\mathcal{I})$ is bijective. 
\end{theorem}
\begin{proof}
Let $\tau(\alpha)$ = $\tau(\alpha')$ with the two cellular representations. One can see $\alpha$ = $\alpha'$ by taking the conjugate of the one of them. The map $\tau$ is surjective. Suppose that $a \in Rep^G(\mathcal{I})$, it turns out to be a cellular representation. The reason is that isotropy groups are identity except vertices \cite{ian2}.
\end{proof}
\indent 

\section{Equivariant Bundles on 1-skeletons}

The space of equivariant bundles $Bun^G_\Gamma(\mathcal{A}^1)$ is classified by means of $Rep^G(\mathcal{I})$ with (\ref{ceq}) and (\ref{deq}). $Rep^G(\mathcal{I})$ is calculated for each finite subgroups of $SO(3).$ 
\begin{theorem} Let $\mathfrak{C_n}$ be a $\mathbb{Z}_n $-equivariant 1-skeleton over $S^2$ with 2 vertices and $n$ edges, $\mathbb{Z}_n $ acting on $\mathfrak{C_n}$ be a cyclic group with the order $n.$ Let $E_n$ be the orbit space of  $\mathfrak{C_n}$ under the group action of $ \mathbb{Z}_n $ with isotropy groupoid $\mathcal{I}_1.$ Then, the following map
$$ {Bun^G_{\mathbb{Z}_n}}(\mathfrak{C_n}) \rightarrow Rep^G(\mathcal{I}_1) $$ is a bijection  and $$  Rep^G(\mathcal{I}_1) \cong \overline{Rep}^G(\mathcal{I}_1) \cong \overline{Hom}(\mathbb{Z}_n,G) \times \overline{Hom}(\mathbb{Z}_n,G).$$
\end{theorem}
\begin{proof} $\mathfrak{C_n}$ is a split $\mathbb{Z}_n$-space, all equivariant bundles are split bundles.
\end{proof}
\begin{theorem} Let $\mathfrak D_n$ be a $D_n$-equivariant \emph{1-skeleton} over $ S^2$
with $2n+2$ vertices, $6n$ edges. Let $D_n$ acting on the CW-complex
$\mathfrak D_n$  be a dihedral group with the order $2n.$ Let $\mathcal D_n$ be an orbit space under the group action $D_n$ with isotropy groupoid $\mathcal{I}_5.$ If $G$ is connected, then there is a bijection
$$ Bun^G_{D_n}(\mathfrak D_n) \rightarrow Rep^G(\mathcal{I}_2) $$ and $$ Rep^G(\mathcal{I}_2) \cong \overline{Rep}^G(\mathcal{I}_2) \cong \overline{Hom}(\mathbb{Z}_2,G) \times \overline{Hom}(\mathbb{Z}_2,G) \times \overline{Hom}(\mathbb{Z}_n,G).$$
\end{theorem}
\begin{proof}
\emph{$\mathfrak D_n$} is split $D_n$-space and $$Bun^G_\Gamma(\mathcal{A}^1) \rightarrow Rep^G(\mathcal{I}_2) \times Bun^G(A)$$ since the group $G$ is a connected compact Lie group, it follows
$$ [S^1,\mathcal{A}^1] \cong \pi_{1}(BG) \cong \pi_{0}(G) \cong 0.$$
\end{proof}
\begin{theorem} Let $\mathfrak{T}$ be a $A_4$-equivariant tetrahedron  and $A_4$ acting on $\mathfrak{T}$ be the tetrahedral group with the order 12. Let $T$ be
an orbit space of $\mathfrak T$ under the group action $A_4$ with isotropy groupoid $\mathcal{I}_3.$ Then, there is a bijection
$$ Bun^G_{A_4}(\mathfrak{T}) \rightarrow Rep^G(\mathcal{I}_3) $$ and $$ Rep^G(\mathcal{I}_3) \cong \overline{Rep}^G(\mathcal{I}_3) \cong \overline{Hom}(\mathbb{Z}_2,G) \times \overline{Hom}(\mathbb{Z}_3,G).$$
\end{theorem}
\begin{proof}
$\mathfrak{T}$ is a split $A_4$-space, all equivariant bundles are split bundle.
\end{proof}
\begin{theorem} Let $\mathcal{C}$ be a $S_4$-equivariant cube and $S_4$ acting on the cube $\mathcal{C}$ be the octahedral group with the order 24. Let $O$ be the orbit space of $\mathcal{C}$ under the group action $S_4$ with isotropy groupoid $\mathcal{I}_4.$ Then, there is a bijection
$$ Bun^G_{S_4}(\mathcal C) \rightarrow Rep^G(\mathcal{I}_4) $$ and $$ Rep^G(\mathcal{I}_4) \cong \overline{Rep}^G(\mathcal{I}_4) \cong \overline{Hom}(\mathbb{Z}_2,G) \times \overline{Hom}(\mathbb{Z}_3,G) \times \overline{Hom}(\mathbb{Z}_4,G).$$
\end{theorem}
\begin{proof}
\emph{ $\mathcal{C}$} is a split $S_4$-space, all equivariant bundles are split bundle.
\end{proof}
\begin{theorem} Let $\mathcal O$ be an $A_5$-equivariant icosahedron and $A_5$ acting on $\mathcal O$ be an icosahedral group with the order 60. Let $I$
be an orbit space of $\mathcal O$ under the group action $A_5$ with isotropy groupoid $\mathcal{I}_5.$ Then, there is a bijection
$$ Bun^G_{A_5} (\mathcal O) \rightarrow Rep^G(\mathcal{I}_5) $$ and $$ Rep^G(\mathcal{I}_5) \cong \overline{Rep}^G(\mathcal{I}_5) \cong \overline{Hom}(\mathbb{Z}_3,G) \times \overline{Hom}(\mathbb{Z}_4,G) \times \overline{Hom}(\mathbb{Z}_5,G).$$
\end{theorem}
\begin{proof}
 $\mathcal O$ is a split $A_5$-space, all equivariant bundles are split bundle.
\end{proof}
\section{$\Gamma$-$G$ Bundles over $S^2$}
\noindent Let $\mathcal{A}^1 \subset S^2$ be a $\Gamma$-equivariant 1-skeleton. Hambleton and Hausmann \cite{ian2} provide an isomorphism between $\Gamma$-equivariant principal $G$-bundles over $\mathcal{A}^1$ and the class of isotropy representations. Additionally, it is shown that a $\Gamma$-equivariant 1-skeleton $\mathcal{A}^1$ is a split $\Gamma$-space in Theorem \ref{splittheorem}. On the other hand, we can construct a cofibration sequence derived from the inclusion map $i:\mathcal{A}^1\rightarrow S^2$ to determine $\Gamma$-equivariant principal $G$-bundles over $S^2$, since the 2-sphere is not $\Gamma$-equivariant. \\
Then, the following sequence
\begin{equation}  \mathcal{A}^1 \xrightarrow{i} S^2 \xrightarrow{j} S^2 \cup C(\mathcal{A}^1) \xrightarrow{k} \Sigma(\mathcal{A}^1) \xrightarrow{\Sigma i} \Sigma(S^2) \xrightarrow{\Sigma j} \Sigma(S^2 \cup C(\mathcal{A}^1)) \rightarrow \cdots \end{equation} is the cofibration of $\Gamma$-equivariant $CW$-complexes where a cone $$C(\mathcal{A}^1)= (\mathcal{A}^1\times[0,1])\big{/} \{ (a,0) \sim  \text{single point}\}$$ and  the  suspension $$\Sigma(\mathcal{A}^1)=(\mathcal{A}^1 \times [-1,1]) \big{/}  \{  (a,-1)\sim \text{single point}, (a,1) \sim \text{single point}\}.$$

\indent The $\Gamma$-fixed set of homotopy classes maps into the space $B(\Gamma,G)$, then the following sequence
\begin{equation} \label{eq:solve} [\Sigma(S^2),Y]_\Gamma \xrightarrow{\Sigma i*} [\Sigma(\mathcal{A}^1),Y]_\Gamma \xrightarrow{k^*} [S^2\cup C(\mathcal{A}^1),Y]_\Gamma \xrightarrow{j^*} [S^2,Y]_\Gamma \xrightarrow{i^*} [\mathcal{A}^1,Y]_\Gamma \end{equation}is the exact sequence of abelian groups provided that $ B(\Gamma,G)= Y= \Omega Z$ is a loop space which is defined by Costenoble and Waner \cite{cost}. Now, $[S^2,Y]$ is determined by $j^*$ and $i^*.$

\vspace{0.1in}
\noindent Topologically, we have homomorphisms  such that  $S^2\cup C(\mathcal{A}^1) \simeq \bigvee S^2$ (induced from 2-cells) and 
$\Sigma(\mathcal{A}^1) \simeq \bigvee S^2$ (induced from 1-cells). 
Then the exact sequence  at (\ref{eq:solve}) turns out to be the following sequence;   
$$ [\bigvee_{1-cells} S^2,Y] \xrightarrow{k^*} [\bigvee_{2-cells} S^2,Y] \xrightarrow{j^*} [S^2,Y] .$$
\begin{equation}  \label {eq:solve1}[\bigvee S^2,Y]_{\Gamma \neq D_n}  \simeq \bigoplus_{1-\chi(\mathcal{A}^1)}\pi_2(BG)\ and\ \Gamma\ acts\ on\ product\ = I\otimes\pi_1(G)\end{equation} as a $\Gamma$-module, or 
\begin{equation}  [\bigvee S^2,Y]_{\Gamma = D_n}  \simeq \bigoplus_{1-\chi(\mathcal{A}^1)}\pi_2(BG)\ and\ \Gamma\ acts\ on\ product\ = (I\oplus \mathbb{Z}\Gamma)\otimes\pi_1(G)\ \end{equation} as a $\Gamma$-module then, we have

\begin{equation} \label {eq:solve2}[\bigvee_{2-cells} S^2,Y]_\Gamma \simeq \bigoplus_\mathcal{N} \pi_2(BG)\ and\ \Gamma\ acts\ on\ product\ = \mathbb{Z}\Gamma\otimes\pi_1(G)\ \end{equation} as a $\Gamma$-module provided that the ideal $I =\mathbb{Z}\{(\gamma-1)\lvert \ \gamma \in \Gamma \} \subset \mathbb{Z}\Gamma.$ The number of copy of $\pi_2(BG)$ at the equation (\ref {eq:solve1}) is calculated by means of counting rotations and order of groups.\\
Let $\mathcal{A}^1 \subset S^2$ be a $\Gamma$-equivariant 1-skeleton. For the cyclic group $\mathbb{Z}_n$, an equivariant 1-skeleton on $S^2$ is constructed by two vertices and $n$ edges. By collapsing an edge to a point, the other edges turn out to be circles. Therefore, we obtain $(n-1)$ circles.\\
\noindent For the dihedral group $D_n$, there is an orbit with 2n elements. By shrinking an edge to a point, we obtain $(2n-1)$ circles. The other 2n orbits with 2 elements are reduced to (2n) circles, by collapsing an edge to a point. Therefore, $(4n-1)$ circles are formed from the group action of the dihedral group.\\
\noindent For the tetrahedral group $A_4$, there are 4 vertex rotations with order 3. By collapsing an edge to a point for each rotation, we obtain 8 circles. There are 3 edge rotations with order 2 that are reduced to 3 circles. In total, we have 11 circles for the tetrahedral group.\\
\noindent For the octahedral group $S_4$, there are 4 vertex rotations with order 3. After collapsing one edge to a point for each rotation, we obtain 8 circles. There are 6 edge rotations with order 2, which yield 6 circles. There are 3 face rotations with order 4, which yield 9 circles. In total, we have 23 circles for the octahedral group.\\
\noindent Finally, for the icosahedral group $A_5$, we count rotations and orders using the same method and obtain 59 circles.\\
\noindent Briefly, we notice that the number of circles for each subgroup $\Gamma$ acting on $\mathcal{A}^1$ is $1-\chi(\mathcal{A}^1)$ for each $\Gamma$-equivariant $\mathcal{A}^1\subset S^2$.\\
\noindent Let $\mathcal{N}$ denote the number of copies of $\pi_2(BG)$ at (\ref{eq:solve2}). It depends on the number of orbits and the order of group $\Gamma$. Except for the dihedral group, finite subgroups of $SO(3)$ have a single orbit. We summarize this in the following table.

\begin{table}[H]
 \centering
\begin{tabular}{l*2|c|r}
Group & $1-\chi(\mathcal{A}^1)$ & $\mathcal{N}$ \\
\hline
Cyclic group & $n-1$ & $n$ \\
Dihedral group & $4n-1$ & $4n$ \\
Tetrahedral group & 11 & 12 \\
Octahedral group & 23 & 24 \\
Icosahedral group & 59 & 60 \\

\end{tabular}
\caption{The number of copy of $\pi_2(BG)$.}

\end{table}

Since $k^*$ is an injective map, depending on $\Gamma$ the following maps hold; 
$$0 \rightarrow I\otimes\pi_1(G) \xrightarrow{k^*} \mathbb{Z}\Gamma\otimes\pi_1(G) \xrightarrow{j^*} \mathbb{Z}\otimes\pi_1(G) $$ or $$0 \rightarrow (I\oplus\mathbb{Z}\Gamma)\otimes\pi_1(G) \xrightarrow{k^*} (\mathbb{Z}\Gamma\oplus\mathbb{Z}\Gamma)\otimes\pi_1(G) \xrightarrow{j^*} \mathbb{Z}\otimes\pi_1(G) $$ 

and
$Coker(k^*) \simeq \mathbb{Z}\otimes\pi_1(G) \simeq [S^2,B(\Gamma,G)]_\Gamma.$
\vspace{.05in}

\noindent Now, let $Z$ and $Y$ be two $\Gamma$-space and $f:Z\rightarrow Y$ be continuous. Define $f^\gamma(z)=\gamma^{-1}f(\gamma z).$ The map $f\rightarrow f^{\gamma}$ gives an action of $\Gamma$ on $[Z,Y]$. Then $f=f^{\gamma}$ $\leftrightarrow$ $\gamma f(z)=f(\gamma z)$ $\leftrightarrow$ $f$ is a $\Gamma$-map. Therefore, we shall say the following $[Z,Y]_\Gamma= Fix(\Gamma,[Z,Y])$ .
\vspace{0.08in}

\noindent $Fix(\Gamma, I)=\{x \in I \lvert \ \gamma x = x\} =0$  and we  determine $Fix(\Gamma, \mathbb{Z}\Gamma ).$   Let  $  t \in \Gamma $ be a generator. $\Gamma$ is acting on  $\mathbb{Z}\Gamma = \mathbb{Z} \oplus \mathbb{Z}t \oplus \cdots \oplus \mathbb{Z}t^{n-1}$  where $\lvert \Gamma \lvert =n$. The fixed set of $\mathbb{Z}\Gamma$ can be determined by $(a_0+a_1t+\cdots +a_{n-1}t^{n-1} )\gamma =(a_0+a_1t+\cdots +a_{n-1}t^{n-1} )$ for all $\gamma \in \Gamma$ implies that fixed elements are $\mathbb{Z}(1+t+\cdots+ t^{n-1})$. Then,  $Fix(\Gamma, \mathbb{Z}\Gamma )= \mathbb{Z}(1\sum_{\gamma \in \Gamma}\gamma).$ Therefore, 
\begin{equation} \label {eq:solve6}[\Sigma (\mathcal{A}^1), Y]_\Gamma = Fix(\Gamma,I\otimes\pi_2Y)=0 \end{equation} and 
\begin{equation} \label {eq:solve7}[S^2\cup C(\mathcal{A}^1),Y]_\Gamma =Fix(\Gamma,\mathbb{Z}\Gamma\otimes \pi_2Y)= \pi_2(Y), \end{equation} since  we have $[S^2,Y]=\pi_2(Y)$ in the sequence (\ref{eq:solve}).   
\begin{theorem} Let $\xi=(E,S^2,p,G,\Gamma)$ be a $\Gamma$-equivariant principal $G$-bundle over $S^2$ with a compact connected Abelian Lie group $G$ and $\Gamma\subset SO(3)$ be a finite subgroup acting on $S^2.$ A $\Gamma$-equivariant principal $G$-bundle over $S^2$ is determined by $Rep^G(\mathcal{I})$ and $c(\xi) \in \pi_2(BG).$
\end{theorem}
\begin{proof} Let $[\nu]$,  $[\xi]$ $\in$ $[S^2,Y]_\Gamma$.  $[S^2,Y]_\Gamma \xrightarrow{i^*} [\mathcal{A}^1,Y]_\Gamma $ and $[\mathcal{A}^1,Y]_\Gamma \cong Rep^G_\Gamma (\mathcal{I}).$ If $Rep^G (\mathcal{I}_\nu ) \ncong Rep^G_\Gamma (\mathcal{I}_\xi)$ then one concludes that they are non-equivariant to each other. If $Rep^G(\mathcal{I}_\nu)\cong Rep^G_\Gamma(\mathcal{I}_\xi)$ then  $$ [\Sigma (\mathcal{A}^1), Y]_\Gamma \rightarrow [S^2\cup C(\mathcal{A}^1),Y]_\Gamma \rightarrow [S^2,Y]_\Gamma$$ and by (\ref{eq:solve6}) and (\ref{eq:solve7}) $$ 0 \rightarrow \pi_2(BG) \xrightarrow{\lvert \Gamma \lvert} [S^2,Y]_\Gamma $$ then this map is reduced to the following congruence.
\end{proof}
\begin{corollary} If $Rep^G(\mathcal{I}_{\xi_1})\cong Rep_\Gamma^G(\mathcal{I}_{\xi_2})$ then $c(\xi_1)\equiv c(\xi_2)$   $mod\  \lvert \Gamma\lvert.$ 
\end{corollary}

\noindent This theorem completes the classification of
equivariant principal bundles over the 2-sphere.  Future studies will be focusing on how we can apply this theorem to the product space $S^2\times S^2$ by these ideas.

\vspace{.05in}

\textbf{Acknowledgements.} I would like to thank I. Hambleton as my supervisor to provide tools of this article. I also would like to thank M. Kalafat for motivating me to publish this article and E. Yalcin for useful discussions.  This work is partially supported by the grant of McMaster University.
\bigskip

\begin{thebibliography}{9999}


\bibitem[1]{bred} G. E.~Bredon, \emph{Introduction  to  compact  transformation  groups}, Pure and Applied Math. vol.~46, Academic Press, 1972.


\bibitem[2]{cost} S. R.~Costenoble and S.~Waner, \emph{Fixed set systems of equivariant infinite loop spaces}, Trans. Amer. Math. Soc., \textbf{326} (1991), 485–-505.

\bibitem[3]{tom} T. tom~Dieck, \emph{Transformation groups}, de Gruyter Studies in Mathematics,
  vol.~8, Walter de Gruyter \& Co., Berlin, 1987.

\bibitem[4]{ian2}  I.~Hambleton and J.~Hausmann,   \emph{Equivariant principal bundles over spheres and cohomogeneity one manifolds}, Proc. London Math. Soc. (3) \textbf{86} (2003), 250–-272.

\bibitem[5]{ian1} I.~Hambleton and J.~Hausmann, \emph{Equivariant bundles and isotropy representations}, Groups Geom. Dyn. \textbf{4} (2010), 127–-162.

\bibitem[6]{kim}  M. K.~Kim, \emph{Classification  of  equivariant  vector bundles  over  two-sphere}, arXiv:math.GR/1005.0681.

\bibitem[7]{lashof1} R. K.~Lashof, \emph{Equivariant bundles}, Illinois J. Math. \textbf{26} (1982),
  257--271.

\bibitem[8]{lashof4} R. K.~Lashof, J. P.~May and G. B.~Segal, \emph{Equivariant bundles with
  abelian structural group}, Proceedings of the Northwestern Homotopy Theory
  Conference (Evanston, Ill., 1982) (Providence, R.I.), Amer. Math. Soc., 1983,
  167--176.

\bibitem[9]{lashof2} R. K.~Lashof and J. P.~May, \emph{Generalized equivariant bundles}, Bull. Soc.
  Math. Belg. S\'er. A \textbf{38} (1986), 265--271.

\bibitem[10]{rees2005notes} E.~Rees, \emph{Notes on Geometry},
Springer Berlin Heidelberg, 2005.


\bibitem[11]{steen} N.~Steenrod,
 \emph{The topology of fibre bundles},
Princeton University Press, 1951.


\bibitem[12]{stewart} T. E.~Stewart, \emph{Lifting group actions in fibre bundles}, Ann. of Math. (2)
  \textbf{74} (1961), 192--198.

\bibitem[13]{verrette2016results}    J. H.~Verrette, \emph{Results  on  algebraic  realization  of  equivariant  bundles  over  the 2-sphere},  PhD thesis, University of Hawai’i at Manoa, 2016.

\bibitem[14]{yalcinkaya2012equivariant}   E.~Yalcinkaya,  \emph{Equivariant principal bundles over the 2-sphere},  Master’s thesis, McMaster University, 2012.







\end{thebibliography}

\bigskip

{\small 
\begin{flushleft} \textsc{Department of Mathematics \& Statistics, McMaster University, Hamilton, ON L8S 4K1, Canada}\\
\textsc{Current Address: Department of Mathematics, University of Rochester,  NY, USA}\\
E-mail: eyup.yalcinkaya@rochester.edu

\end{flushleft}
}

\end{document}